\documentclass{article}
\usepackage{fullpage}
\usepackage{amssymb,amsmath,amsthm}
\usepackage{graphicx} 
\usepackage{xcolor}
\usepackage[unicode, hidelinks]{hyperref}
\usepackage[noabbrev,capitalize]{cleveref}

\title{On the structure of extremal point-line arrangements}
\author{Gabriel Currier \footnote{Department of Mathematics, University of British Columbia, Vancouver, BC, Canada. Emails: currierg@math.ubc.ca, solymosi@math.ubc.ca}
\and Jozsef Solymosi \footnotemark[1]
\and Hung-Hsun Hans Yu \footnote{Department of Mathematics, Princeton University, Princeton, NJ 08544. Email: hansonyu@princeton.edu}}

\date{\today}

\newtheorem{theorem}{Theorem}[section]
\newtheorem{proposition}[theorem]{Proposition}

\newtheorem{definition}[theorem]{Definition}
\newtheorem{corollary}[theorem]{Corollary}

\providecommand{\RR}{\mathbb{R}}
\providecommand{\CC}{\mathbb{C}}

\providecommand{\cK}{\mathcal{K}}

\crefname{equation}{}{}

\begin{document}

\maketitle
\begin{abstract}
    In this note, we show that extremal Szemer\'edi--Trotter configurations are rigid in the following sense: If $P,L$ are sets of points and lines determining at least $C|P|^{2/3}|L|^{2/3}$ incidences, then there exists a collection $P'$ of points of size at most $k = k_0(C)$ such that, heuristically, fixing those points fixes a positive fraction of the arrangement. That is, the incidence structure and a small number of points determine a large part of the arrangement.
    The key tools we use are the Guth--Katz polynomial partitioning, and also a result of Dvir, Garg, Oliveira and Solymosi that was used to show the rigidity of near-Sylvester--Gallai configurations.
\end{abstract}
\section{Introduction}

Let $P$ be a set of points and $L$ a set of lines in $\RR^2$. A pair $(p,\ell) \in P \times L$ is defined to be an \emph{incidence} if the point $p$ is on the line $\ell$, and we let $I(P,L)$ denote the set of incidences between $P$ and $L$. The following classical result of Szemer\'edi and Trotter gives a bound on the maximum number of incidences among $n$ points and $m$ lines.

\begin{theorem}[Szemer\'edi--Trotter, 1984]\label{thm:szt}
Let $P$ be a set of $n$ points and $L$ be a set of $m$ lines in $\RR^2$. Then,  $$|I(P,L)| \lesssim n^{2/3}m^{2/3} + n + m \footnotemark[3]\footnotetext[3]{We say that $A \lesssim B$ if there exists a constant $C$ such that $A \le CB$, and $A \approx B$ if $A \lesssim B$ and $B \lesssim A$. The constants are assumed to be absolute unless otherwise specified.}$$
\end{theorem}

The ``interesting'' range is considered to be $m^{1/2} \lesssim n \lesssim m^2$, where the $n^{2/3}m^{2/3}$ term dominates. Outside of this range, the proof of the theorem is quite straightforward, and the extremal configurations are essentially trivial. Despite a wealth of literature and many applications of \cref{thm:szt}, very little is known about the configurations of points and lines that achieve $\approx n^{2/3}m^{2/3}$ incidences. We direct the interested reader to \cite{Elekes1997,Elekes1999,PetridisRocheNewtonRudnevWarren2022,RudnevShkredov2022,ShefferSilier2023,Solymosi2006} for an introduction to this area. 

Recently, Katz and Silier \cite{KatzSilier2023} proved a structural result about configurations of $n$ points and $n$ lines determining $\approx n^{4/3}$ incidences. They showed that such configurations must admit a particular type of cell decomposition, and used this to show that any such configuration can be encoded using only $n^{1/3}$ parameters. 
Our main result in this paper is that given the additional information of the incidence structure $I(P,L)$ of such a configuration, a constant number of parameters, in fact, suffices to describe the entire arrangement. We will make this rigorous using the following notion of rigidity.

\begin{definition}
Let $T$ be a finite collection of triples in some finite ground sets $S$, and let $\mathcal{K}_T\subseteq (\CC^2)^S$ be the variety consisting of assignments $(p_s)_{s\in S}$ where each $p_s$ is a point in $\CC^2$, so that if $\{s_1,s_2,s_3\}\in T$, then $p_{s_1},p_{s_2},p_{s_3}$ are collinear.
Let $P = (p_s)_{s\in S} \in \cK_T$.
We say that $P$ is $r$-rigid with respect to $T$ if the dimensions of the irreducible components through $P$ in $\cK_T$ are at most $r$.
\end{definition}

Similarly, we say a configuration of points and lines $(P,L)$ is $r$-rigid if $P$ is $r$-rigid with respect to $T_L$ where $T_L\subseteq \binom{P}{3}$ is the collection of collinear triples determined by $L$. One can think of rigidity in this sense as the number of ``degrees of freedom'' that a given collection of points has, given certain collinearity conditions. Projective transformations preserve collinearity and provide $8$ natural degrees of freedom, and thus, any configuration is at least $8$-rigid. This definition of rigidity was introduced in \cite{DvirGargOliveiraSolymosi2018}, where they showed that point configurations containing many collinear triples are rigid. Our goal is to prove the following theorem: that is, asymptotically sharp Szemer\'edi--Trotter configurations are rigid. 

\begin{theorem}\label{thm:main}
For every constant $\Delta>0$, there exists a constant $C_1\geq 1$ depending only on $\Delta$ such that the following is true.
Let $P$ be a set of $n$ points and $L$ a set of $m$ lines in $\RR^2$, with $m \le n^{2}/C_1$ and $n\leq m^2/C_1$. Furthermore, suppose that $P$ and $L$ determine at least $\Delta n^{2/3}m^{2/3}$ incidences.
Then there exists a subset $P' \subset P$ of size $\Omega_{\Delta} (n)$ such that $(P',L)$ is $O_{\Delta}(1)$-rigid.
\end{theorem}

We view this result as a step toward characterizing all near-optimal configurations of the Szemer\'edi--Trotter theorem. 
We hope to eventually be able to recover a large portion of the configurations even without the full data of the incidence structure, which would be one path towards characterizing near-optimal configurations.

Our proof strategy is as follows.
First, using the Guth--Katz polynomial partitioning theorem \cite{GuthKatz15}, we divide the configurations into many cells with a polynomial.
Then, we apply a result of Dvir, Garg, Oliveira and Solymosi \cite{DvirGargOliveiraSolymosi2018} to show that most cells are rigid.
Lastly, we show that a small number of cells already determine a large part of the configuration.
A more detailed proof strategy is outlined at the beginning of \cref{sec:main-proof}.

\section{Preliminaries}

We will need a few preliminary results before commencing with the proof of \cref{thm:main}. The first is an auxiliary lemma we will need to easily work with the notion of rigidity.

\begin{proposition}\label{prop:fullrigid}
Let $S$ be a finite set and $S_1,\ldots, S_{\ell}$ be some disjoint subsets of $S$.
Let $T\subseteq \binom{S}{3}$ be a collection of triples, and let $P=(p_s)_{s\in S}\in \cK_T$ be a collection of points in $\CC^2$.
For each $i\in [\ell]$, let $T_i= T\cap \binom{S_i}{3}$ and $P_i = (p_s)_{s\in S_i}$.
Suppose that $P_i$ is $r$-rigid with respect to $T_i$ for each $i$, and that the only $P'\in \cK_{T}\subseteq (\CC^2)^{S}$ with $P_i' = P_i$ for all $i\in[\ell]$ is $P'=P$.
Then, $P$ is $(\ell r)$-rigid with respect to $T$.
\end{proposition}
\begin{proof}
    Let $\Tilde{S}$ be the union of $S_1,\ldots, S_{\ell}$ and $\Tilde{T}$ be the union of $T_1,\ldots, T_{\ell}$ that can be thought of as triples in $\Tilde{S}$.
    Then $\mathcal{K}_{\Tilde{T}} = \prod_{i=1}^{\ell}\mathcal{K}_{T_i}$.
    Let $\pi:(\CC^2)^S\to (\CC^2)^{\Tilde{S}}$ be the canonical projection.
    
    For any component $V$ of $\cK_T$ through $P$, we would like to compare the dimension of $V$ to the dimension of $\overline{\pi(V)}$.
    Note that there is some dense Zariski open $U\subseteq \overline{\pi(V)}$ such that $(\pi|_V)^{-1}(P')$ is of pure dimension $\dim V-\dim \overline{\pi(V)}$ for every $P'\in U$.
    By the upper semicontinuity of the dimensions of fibers on the source, we know that for every $P'\in V$, the fiber of $\pi|_V$ above $\pi(P')$ is locally of dimension at least $\dim V-\dim \overline{\pi(V)}$ at $P'$.
    However, the fiber of $\pi|_V$ above $\pi(P)$ is a single point by the assumption, showing that $\dim V-\dim \overline{\pi(V)}\leq 0$.
    Therefore $\dim V = \dim \overline{\pi(V)}$.
    Since $\overline{\pi(V)}\subseteq \mathcal{K}_{\Tilde{T}} = \prod_{i=1}^{\ell}\mathcal{K}_{T_i}$ is irreducible and contains $(P_1,\ldots,P_\ell)$, we have that $\dim \overline{\pi(V)}\leq \ell r$.
    This shows the claim.
\end{proof}

The next proposition allows us to deduce the rigidity of points from the rigidity of lines.

\begin{proposition}\label{prop:dualrigidity}
    Let $P$ be a set of points, $L$ be a set of lines in $\CC^2$ so that every point in $P$ lies on at least two lines in $L$.
    Let $r$ be a positive integer, and $T^*$ be the concurrent triplets of $L$ determined by $P$.
    Note that by taking the dual with respect to a general point in $\CC^2$, we can map $L$ to a collection of points $d(L)$ in $\CC^2$.
    If $d(L)$ is $r$-rigid with respect to $T^*$, then $(P,L)$ is $r$-rigid as well.
\end{proposition}
\begin{proof}
    For any point $p$, let $d(p)$ be its dual line, and let $d(\ell)$ be the dual point of any line $\ell$.
    For any point $p\in P$, choose two lines $\ell_{p,1}$ and $\ell_{p,2}$ in $L$ that pass through it.
    Consider the rational map $f:(\CC^2)^{L}\dashrightarrow (\CC^2)^{P}$ defined as 
    \[f((x_{\ell})_{\ell\in L}) = \left(d\left(x_{\ell_{p,1}}\right)\cap d\left(x_{\ell_{p,2}}\right)\right)_{p\in P}.\]
    It is clear that $d(L)$ is in the domain of $f$ and $f(d(L)) = P$.
    
    We can also consider another rational map $g:(\CC^2)^{P}\dashrightarrow (\CC^2)^{L}$ where $g((y_p)_{p\in P})$ is defined as follows.
    For every line $\ell\in L$, fix two general points $u_{\ell},v_{\ell}$ in $\ell$.
    For every line $\ell\in L$ that passes through at least two points $P$, choose two points $p_{\ell,1}, p_{\ell,2}$ in $P$ lying on $L$ and set $z_{\ell,1} = y_{p_{\ell,1}}$ and $z_{\ell,2} = y_{p_{\ell,2}}$.
    For $\ell\in L$ that passes through exactly one point in $P$, let $p_{\ell,1}$ be that unique point, $z_{\ell,1} = y_{p_{\ell,1}}$ and $z_{\ell,2} = v_{\ell}$.
    Finally, if there are no points on $\ell$ in $P$, then set $z_{\ell,1} = u_{\ell}, z_{\ell,2} = v_{\ell}$.
    We now define $g((y_p)_{p\in P})_{\ell}$ as the dual point of the line connecting $z_{\ell,1}$ and $z_{\ell,2}$.
    If $u_{\ell}$ and $v_{\ell}$ are chosen appropriately, then for $(y_p)_{p\in P}=P$ we have $z_{\ell,1}\neq z_{\ell,2}$ and thus $P$ lies in the domain of $g$.
    Moreover, it is clear that $g(P) = d(L)$.

    Let $T\subseteq \binom{P}{3}$ be the collinear triplets determined by $L$.
    We now show that for any $y$ in $\cK_T$ and also in the domain of $g$, we have $g(y)\in \cK_{T^*}$ and $f(g(y))=y$.
    To check this, it suffices to check that for every $\ell_1,\ell_2,\ell_3\in L$ passing through $p\in P$, we have $g(y)_{\ell_1}, g(y)_{\ell_2}, g(y)_{\ell_3}$ represent three concurrent lines meeting each other at $y_p$.
    Now suppose that for each $i=1,2,3$, we need to show that $y_p$ lies on $d\left(g(y)_{\ell_i}\right)$.
    If $p\in \{p_{\ell_i,1},p_{\ell_i,2}\}$, then there is nothing to check.
    Otherwise, we must have that $p_{\ell_i,1},p_{\ell_i,2}\in P$, and so $y_p, y_{p_{\ell_i,1}}, y_{p_{\ell_i,2}}$ are collinear as $y\in \cK_T$, showing that $y_p$ is on the line $d\left(g(y)_{\ell_i}\right)$ as well.
    
    Finally, for every irreducible component $V$ of $\cK_{T}$ that passes through $P$, we know that $g$ is a rational map on $V$ as well because $P\in V$, and thus $\overline{g(V)}$ lies in some irreducible component $U$ in $\cK_{T^*}$ passing through $d(L)$.
    As $f\circ g$ is identity on $V$, we have that  $\dim V=\dim \overline{f(g(V))}\leq \dim \overline{g(V)}\leq \dim U\leq r$ by the rigidity of $d(L)$.
    This shows that $P$ is $r$-rigid with respect to $T$ as well, as desired.
\end{proof}

Next, we will need the following (special case of a) result of Dvir, Garg, Oliveira and Solymosi \cite[Theorem 4.1]{DvirGargOliveiraSolymosi2018}, which gives the rigidity of point-line configurations containing many collinear triple lines.

\begin{theorem}[Dvir--Garg--Oliveira--Solymosi]\label{thm:solydvir}
Let $S$ be a finite set, $P=(p_s)_{s\in S}$ be a configuration of points $\RR^2$  indexed by $S$, and $T$ a collection of triples in $S$. 
Suppose furthermore that $P$ is a non-singular point of $\mathcal{K}_T$ and that:
\begin{enumerate}
\item Each $s\in S$ is in at least $k$ triples in $T$.
\item Each pair of elments in $S$ is in at most $t$ triples in $T$.
\item For any line $\ell\subseteq\RR^2$ and any $s$ with $p_s\in \ell$, there are at most $k/2$ triples $\{s,s',s''\}\in T$ such that $p_{s'},p_{s''}$ are also in $\ell$.
\end{enumerate}
Then, $P$ is $\lfloor \frac{8tn}{4t + k}\rfloor$-rigid with respect to $T$.
\end{theorem}

In fact, the condition that $P$ is a non-singular point of $\mathcal{K}_T$ can be removed, which we record as the following corollary.

\begin{corollary}\label{cor:solydvir}
    \cref{thm:solydvir} holds even if $P$ is a singular point of $\mathcal{K}_T$.
\end{corollary}
\begin{proof}
    It suffices to show that for any irreducible component $V$ in $\mathcal{K}_T$ that contains $P$, we have $\dim V$ is at most $\frac{8tn}{4t + k}$.
    Since $\mathcal{K}_T$ is a variety, the component $V$ is reduced, and thus the smooth points of $V$ are dense in $V$.
    Moreover, as $P$ satisfies Condition 3 and Condition 3 is an open condition (in Zariski topology), the points that satisfy Condition 3 are dense in $V$.
    Therefore there exists a smooth point $P'$ of $V$ satisfying Condition 3 that does not lie in any other components of $\mathcal{K}_T$.
    Then $P'$ is a smooth point of $\mathcal{K}_T$ where all three conditions are met.
    The theorem shows that $P'$ is  $\lfloor \frac{8tn}{4t + k}\rfloor$-rigid with respect to $T$, and in particular $\dim V \leq \frac{8tn}{4t + k}$, as desired.
\end{proof}

From this, we can derive the following useful corollary. We say that points $x,y,z$ are a consecutive collinear triple in $P$ if they are collinear (in that order), and there are no other points from $P$ between $x$ and $z$.

\begin{corollary}\label{cor:coltrip}
Let $P$ be a collection of $n$ points in $\RR^2$ with $Cn^2$ consecutive collinear triples. Then, there exists $P' \subset P$ such that $|P'| \gtrsim n$ and $P'$ determines $\gtrsim n^2$ consecutive collinear triples, and also $P'$ and its collinear triples are $r$-rigid, where $r$ is a constant depending only on $C$.
\end{corollary}

\begin{proof}
We perform the following process: one-by-one we discard points if they are contained in fewer than $Cn/2$ triples; furthermore, when we discard a point, we discard all triples that contain it as well. Since there are only $n$ points, we can lose at most $Cn^2/2$ triples in this way and are left at the end still with at least $Cn^2/2$. Let $P'$ be the set of remaining points. In $P'$, each triple contains three pairs of points, and each pair of points is in at most two triples. Thus, the number of pairs of points in $P'$ is at least $Cn^2/2$ as well, showing that $|P'| \gtrsim n$.

Now, since each point in $P'$ is in at least $Cn/2$ triples, $P'$ satisfies \cref{thm:solydvir} with $k = Cn/2$ and $t = 2$. Therefore, $P'$ is $$\frac{16|P'|}{(8+Cn/2)} \le \frac{32}{C}$$
rigid. 
\end{proof}

Before we start the proof of our main theorem, we will need a partitioning theorem and associated tools. For convenience, we will use a combination of polynomial partitioning and B\'ezout's theorem.

\begin{theorem}[{see \cite[Theorem 4.1]{GuthKatz15}}]\label{thm:polypart}
Let $P$ be a set of $n$ points in $\RR^2$. Then, for any positive integer $D$, there exists a polynomial $F$ of degree at most $D$ such that there are at most $D^2$ cells in $\RR^2$ determined by $Z(F)$ and each cell has at most $O(n/D^2)$ points from $P$.
\end{theorem}

\begin{theorem}[B\'ezout's Theorem]\label{thm:bezout}
Let $F_1,F_2$ be algebraic plane curves of degree $d_1,d_2$ in $\RR^2$. If $F_1,F_2$ share no common component, then they have at most $d_1d_2$ common zeroes.
\end{theorem}

Finally, we will repeatedly use the following simple corollary of \cref{thm:szt} in various cleaning steps.

\begin{proposition}\label{prop:cleaningline}
    Let $\alpha,\beta$ be two constants with $2\alpha +\beta = 1$ and $\alpha \leq 2/3$.
    Let $C_1>0$ be a constant and $m,n$ be positive integers such that $m\leq n^2/C_1$ and $n\leq m^2/C_1$.
    Assume that $P$ is a set of points and $L$ is a set of lines in $\RR^2$ with $|P|\lesssim m^{\alpha}n^{\beta}$ and $|I(P,L)|\gtrsim m^{\alpha+\frac{2}{3}}n^{\beta-\frac{1}{3}}$. Then, $|L|\gtrsim m^{\frac{1}{2}\alpha+1}n^{\frac{1}{2}\beta-\frac{1}{2}}$ as long as $C_1$ is sufficiently large with respect to all implicit constants.
\end{proposition}
\begin{proof}
    By \cref{thm:szt}, we have
    \[m^{\alpha+\frac{2}{3}}n^{\beta-\frac{1}{3}}\lesssim |I(P,L)|\lesssim m^{\alpha}n^{\beta}+|L|+m^{\frac{2}{3}\alpha}n^{\frac{2}{3}\beta}|L|^{2/3}.\]
    Note that
    \[m^{\alpha}n^{\beta}=m^{\alpha+\frac{2}{3}}n^{\beta-\frac{1}{3}}\left(\frac{n}{m^2}\right)^{1/3}\leq m^{\alpha+\frac{2}{3}}n^{\beta-\frac{1}{3}}C_1^{-1/3},\]
    so as long as $C_1$ is sufficiently large, we must have that $\max\left(|L|, m^{\frac{2}{3}\alpha}n^{\frac{2}{3}\beta}|L|^{2/3}\right)\gtrsim m^{\alpha+\frac{2}{3}} n^{\beta-\frac{1}{3}}$.

    We deal with the case $|L|\gtrsim m^{\alpha+\frac{2}{3}} n^{\beta-\frac{1}{3}}$ first.
    Notice that $2\left(\frac{1}{2}\alpha-\frac{1}{3}\right) = \alpha-\frac{2}{3} = \frac{1-\beta}{2}-\frac{2}{3} = -\left(\frac{1}{2}\beta+\frac{1}{6}\right)$.
    Therefore
    \[m^{\alpha+\frac{2}{3}} n^{\beta-\frac{1}{3}}=m^{\frac{1}{2}\alpha-\frac{1}{3}}n^{\frac{1}{2}\beta+\frac{1}{6}}\cdot m^{\frac{1}{2}\alpha+1}n^{\frac{1}{2}\beta-\frac{1}{2}}\geq C_1^{-\left(\frac{1}{2}\alpha-\frac{1}{3}\right)}m^{\frac{1}{2}\alpha+1}n^{\frac{1}{2}\beta-\frac{1}{2}}\geq m^{\frac{1}{2}\alpha+1}n^{\frac{1}{2}\beta-\frac{1}{2}}\]
    as long as $C_1\geq 1$ since $\frac{1}{2}\alpha-\frac{1}{3}\leq 0$.
    This shows that $|L|\gtrsim m^{\frac{1}{2}\alpha+1}n^{\frac{1}{2}\beta-\frac{1}{2}}$ as well.
    In the second case where $m^{\frac{2}{3}\alpha}n^{\frac{2}{3}\beta}|L|^{2/3}\gtrsim m^{\alpha+\frac{2}{3}} n^{\beta-\frac{1}{3}}$, we get that
    \[|L|\gtrsim \left(m^{\frac{1}{3}\alpha+\frac{2}{3}}n^{\frac{1}{3}\beta-\frac{1}{3}}\right)^{3/2} = m^{\frac{1}{2}\alpha+1}n^{\frac{1}{2}\beta-\frac{1}{2}},\]
    which is exactly what we want.
\end{proof}

As the roles of points and lines are symmetric, we also have the following.
\begin{proposition}\label{prop:cleaningpoint}
    Let $\alpha,\beta$ be two constants with $\alpha +2\beta = 1$ and $\beta \leq 2/3$.
    Let $C_1>0$ be a constant and $m,n$ be positive integers such that $m\leq n^2/C_1$ and $n\leq m^2/C_1$.
    Assume that $P$ is a set of points, $L$ is a set of lines in $\RR^2$ with $|L|\lesssim m^{\alpha}n^{\beta}$ and $|I(P,L)|\gtrsim m^{\alpha-\frac{1}{3}}n^{\beta+\frac{2}{3}}$. Then, $|P|\gtrsim m^{\frac{1}{2}\alpha-\frac{1}{2}}n^{\frac{1}{2}\beta+1}$ as long as $C_1$ is sufficiently large with respect to all implicit constants.
\end{proposition}

\section{Proof of the Main Theorem}\label{sec:main-proof}

Before we begin the proof of \cref{thm:main}, we give a brief outline of the argument. For the sake of this outline, we assume there are about $n$ points and $n$ lines.

\begin{itemize}
    \item We create a partitioning of the pointset using \cref{thm:polypart}, and show that within many cells, there are large pointsets that are at most $r$-rigid, where $r$ is a constant depending only on $\Delta$.
    \item We pick a large constant (depending only on $\Delta$) number of these cells, and show that there must be about $n^{2/3}$ lines determined by the points in these cells.
    \item We take the $n^{2/3}$ lines, and using the incidence structure of $P$ and $L$, we show that there must be many points (i.e. more than $n^{2/3}$) whose positions are uniquely determined by these lines.
    \item We repeat the previous step, using the points recovered in the previous step to recover the position of even more lines. Iterating this process, alternating with points and lines, we recover as many points and lines as possible. We show that at this point, we must have recovered the positions of about $n$ points, up to a multiplicative constant.
    \item Since we used only a constant number of cells and each was $r$-rigid, we use Proposition \ref{prop:fullrigid} to conclude that our arrangement must be at most $C$-rigid, where $C$ is a constant depending only on $\Delta$.
\end{itemize}

In actuality, the proof will be a bit different; there will be a number of pruning steps, and the iteration in the final steps will appear as one step in the actual proof, for simplcity. However, we encourage the readers to keep in mind the above outline as they read through the argument.
\begin{proof}[Proof of \cref{thm:main}]
For the remainder of the proof, all implied constants are allowed to depend on $\Delta$, and only other dependencies will be specified.
Throughout the process, we will choose constants $C_1,C_2,C_3>0$ and a positive integer $k$.
The readers should think of $C_2$ as small, $C_3$ as a small constant depending on $C_2$, $k$ sufficiently large with respect to $C_3$, and finally, $C_1$ sufficiently large with respect to all the other constants.

Let $P,L$ be collections of $n$ points and $m$ lines that determine $\Delta n^{2/3}m^{2/3}$ incidences. 
We will show a stronger statement: there exists $P^*\subseteq P$ with size at least $\gtrsim n$ such that $(P^*,L)$ is $O(1)$-rigid, and also $|I(P^*,L)|\gtrsim m^{2/3}n^{2/3}$.
Given this, we will make a simplifying assumption that $m\geq n$.
Here we briefly explain how to prove the theorem in generality assuming we have proved the satement for $m\geq n$.
To deal with the case $m<n$, we take the dual and run the entire argument.
What we end up with is a subset $L^*\subseteq L$ of size $\gtrsim m$ such that $d(L^*)$ is $O(1)$-rigid with respect to $T^*$, where $T^*\subseteq \binom{L^*}{3}$ is the concurrent triplets determined by $P$.
We also know that $P$ and $L^*$ determine at least $\gtrsim m^{2/3}n^{2/3}$ incidences.
Let $P^*$ be the set of points in $P$ that are on at least $2$ lines in $L^*$.
As $P$ and $L^*$ determines at least $\gtrsim m^{2/3}n^{2/3}$ incidences and we lose at most $n$ incidences by removing all points not in $P^*$, we also have $\gtrsim m^{2/3}n^{2/3}$ incidences between $P^*$ and $L^*$ as long as $C_1$ is sufficiently large: this is because $n\leq m^{2/3}n^{2/3}C_1^{-1/3}$ by the assumption.
Applying \cref{prop:cleaningpoint} with $(\alpha,\beta) = (1,0)$, we get that $|P^*|\gtrsim n$.
It is also clear that any point in $P\backslash P^*$ does not contribute to any concurrent triplets in $T^*$.
Now, by \cref{prop:dualrigidity} applied to $P^*$ and $L^*$, we get that $(P^*,L^*)$ is $O(1)$-rigid, as desired.

From now on, we will assume that $m\geq n$.
\paragraph{Step 0: Cleaning.} We will first clean the sets $P$ and $L$ up so we can put more assumptions safely.
Iteratively throw away points and lines that are less than $\frac{\Delta m^{2/3}}{4n^{1/3}}$- and $\frac{\Delta n^{2/3}}{4m^{1/3}}$-rich respectively. At the end of this process, we are left with a collection of points $P'$ and lines $L'$ such that each point in $P'$ is at least $\frac{\Delta m^{2/3}}{4n^{1/3}}$-rich, each line in $L'$ is at least $\frac{\Delta n^{2/3}}{4m^{1/3}}$-rich, and $(P',L')$ still determine at least $\frac{\Delta}{2}n^{2/3}m^{2/3}$ incidences. 
By \cref{prop:cleaningline} applied to $P', L'$ with $(\alpha,\beta) = (0,1)$, we get that $|L'|\gtrsim m$ as long as $C_1$ is sufficiently large with respect to $\Delta$.
By \cref{prop:cleaningpoint} applied to $P', L'$ with $(\alpha,\beta) = (1,0)$, we also get that $|P'|\gtrsim n$.

\paragraph{Step 1: Polynomial partitioning.} In this step, we will partition the plane into cells and study each cell individually.
Apply \cref{thm:polypart} to $P'$ with $D = \frac{\Delta n^{2/3}}{16m^{1/3}}$ to get a partitioning polynomial $F$ with at most $\frac{\Delta^2 n^{4/3}}{2^8m^{2/3}}$ cells, each having $ \lesssim \frac{m^{2/3}}{n^{1/3}}$ points from $P'$. 
We say now that a triple from $x,y,z \in P'$ is consecutive in-cell collinear if $x,y,z$ are collinear (in that order), no other points from $P'$ are between $x$ and $z$, and $x,y,z$ are all in the same cell defined by $Z(F)$. We make a few notes.

\begin{enumerate}
\item At most $D = \frac{\Delta n^{2/3}}{16m^{1/3}} \leq \frac{\Delta}{16C_1^{2/3}}m\leq |L'|/2$ lines from $L'$ can be contained entirely in $Z(F)$ as long as $C_1$ is sufficiently large.
\item By \cref{thm:bezout}, any line $\ell$ not contained in $Z(F)$ can intersect $Z(F)$ at most $D = \frac{\Delta n^{2/3}}{16m^{1/3}}$ times.  Each intersection with $Z(F)$ can interrupt at most $3$ collinear triples and $\ell$ is $\frac{\Delta n^{2/3}}{4m^{1/3}}$-rich, and thus must contain at least $\left(\frac{\Delta n^{2/3}}{4m^{1/3}}-2\right)$ collinear consecutive triples. Combining these facts, we see that $\ell$ must contain at least $\frac{\Delta n^{2/3}}{8m^{1/3}}$ consecutive in-cell collinear triples.
\item Summing over all lines not in $Z(F)$, we see that there must be at least $|L'|\frac{\Delta n^{2/3}}{16m^{1/3}} \gtrsim n^{2/3}{m^{2/3}}$ total consecutive in-cell collinear triples determined by lines in $|L'|$.
\end{enumerate} 

\paragraph{Step 2: Apply \cref{cor:solydvir} to good cells.} Now we will apply \cref{cor:solydvir} in cells that satisfy certain properties, and then we will show that there are many cells with this property.

We say that a cell is \emph{good} if it contains at least $\frac{C_2m^{4/3}}{n^{2/3}}$ consecutive collinear triples, where $C_2 > 0$ is a small constant to be chosen later. Cells that are not good are referred to as \emph{bad}.
Note that in any good cell, there are $\frac{C_2m^{4/3}}{n^{2/3}}$ consecutive collinear triples, and $\lesssim \frac{m^{2/3}}{n^{1/3}}$ points. Thus, by \cref{cor:coltrip}, in any good cell $D$ there exists a collection of points $P_D$ of size $\approx_{C_2} \frac{m^{2/3}}{n^{1/3}}$ still determining $\gtrsim_{C_2} \frac{m^{4/3}}{n^{2/3}}$ collinear triples, such that $P_D$ is $r_0$-rigid, for a constant $r_0$ independent of $n$ and $m$. We let $L_D$ denote the set of lines determined by consecutive collinear triples in $P_D$. We note that the number of incidences between $P_D$ and $L_D$ must be at least as many as the number of triples. If $C_1$ is sufficiently large with respect to $C_2$ and $\Delta$, by \cref{prop:cleaningline} applied to $P_D$ and $L_D$ with $(\alpha,\beta) = \left(\frac{2}{3},-\frac{1}{3}\right)$, we see that $|L_D|\ge \frac{C_3m^{4/3}}{n^{2/3}}$ where $C_3 > 0$ only depends on $C_2$ and $\Delta$.

Now we show that there are $\gtrsim \frac{n^{4/3}}{m^{2/3}}$ good cells if $C_2$ is sufficiently small.
As there are at most $\frac{\Delta^2n^{4/3}}{2^8m^{2/3}}$ cells, bad cells can account for at most $$\left(\frac{C_2m^{4/3}}{n^{2/3}}\right)\left(\frac{\Delta^2 n^{4/3}}{2^8m^{2/3}}\right) = \frac{C_2\Delta^2}{2^8}n^{2/3}m^{2/3}$$ of the total consecutive in-cell collinear triples. Since there are at least $\gtrsim n^{2/3}{m^{2/3}}$ consecutive in-cell collinear triples determined by $L'$, if we choose $C_2$ to be small enough, then at least $\gtrsim n^{2/3}{m^{2/3}}$ are in good cells. A given cell contains $\lesssim \frac{m^{2/3}}{n^{1/3}}$ points, which producce at most $ \lesssim \frac{m^{4/3}}{n^{2/3}}$ consecutive collinear triples.
Hence there must be $\gtrsim \frac{n^{4/3}}{m^{2/3}}$ good cells.
\paragraph{Step 3: Find some cells determining many lines.} In this step, we will show that, by taking some large (constant) number of cells, the union of the $L_D$'s can be large as well. 
\paragraph{Claim 1:} For each positive integer $k$, there exists a collection of good cells $D_1,\dots,D_k$ such that 
$$\left|\bigcup L_{D_i}\right| \ge \frac{kC_3}{2}\left(\frac{m^{4/3}}{n^{2/3}}\right),$$ assuming $C_1$ is sufficiently large with respect to $C_2,C_3$ and $k$.
\begin{proof}
We will proceed by induction on $k$; for the $k=1$ case, we may simply select any good cell. For larger $k$, by induction we select good cells $D_1,\dots,D_{k-1}$ with $\left|\bigcup_{i=1}^{k-1} L_{D_i}\right| \ge \frac{(k-1)C_3}{2}\left(\frac{m^{4/3}}{n^{2/3}}\right)$. If $\left|\bigcup_{i=1}^{k-1} L_{D_i}\right|$ is at least $\frac{kC_3}{2}\left(\frac{m^{4/3}}{n^{2/3}}\right)$, then we let $D_k$ be an arbitrary good cell, and the inductive step is complete. If not, then we count the number of pairs $(D,\ell)$ where $D$ is a good cell and $\ell \in L_D$. We know there are $\gtrsim \frac{n^{4/3}}{m^{2/3}}$ good cells, and each has $\gtrsim_{C_2}\frac{m^{4/3}}{n^{2/3}}$ lines contributing a triple. Thus, there are $\gtrsim_{C_2} n^{2/3}m^{2/3}$ such pairs. Of the lines from $\bigcup_{i=1}^{k-1} L_{D_i}$, each can enter at most $\frac{\Delta n^{2/3}}{16m^{1/3}}$ different cells. Thus, the lines we have selected so far can account for at most $\frac{\Delta n^{2/3}}{16m^{1/3}} \frac{kC_3}{2}\left(\frac{m^{4/3}}{n^{2/3}}\right) \lesssim_{k,C_3} m$ of these pairs. 
As $n^{2/3}m^{2/3}\geq C_1^{1/3}m$, when $C_1$ is sufficiently large with respect to $C_2,C_3$ and $k$, we conclude that there must be a good cell $D_k$ with $\gtrsim_{C_2}\frac{n^{2/3}m^{2/3}}{n^{4/3}m^{-2/3}}=\frac{m^{4/3}}{n^{2/3}}$ lines that are not in $\bigcup_{i=1}^{k-1} L_{D_i}$. 
When $C_3$ is sufficiently small with respect to $C_2$, the quantity is at least $\frac{C_3}{2}\left(\frac{m^{4/3}}{n^{2/3}}\right).$
Then

$$\left|\bigcup_{i=1}^{k} L_{D_i}\right| \ge \frac{(k-1)C_3}{2}\left(\frac{m^{4/3}}{n^{2/3}}\right) + \frac{C_3}{2}\left(\frac{m^{4/3}}{n^{2/3}}\right) \ge \frac{kC_3}{2}\left(\frac{m^{4/3}}{n^{2/3}}\right)$$ 
which completes the proof of the claim.
\end{proof}

\paragraph{Step 4: Show that those cells determine most lines.} Let $\mathcal{D}$ be a collection of good cells given by Claim $1$, for $k$ a large integer to be chosen later. Furthermore, we let 
$L_{\mathcal{D}} := \bigcup_{D \in \mathcal{D}} L_D$ and $P_{\mathcal{D}}:= \bigcup_{D \in \mathcal{D}} P_D$.
In the final step, we will show that $L_{\mathcal{D}}$ already determines most of the lines in $L$, and the desired conclusion will follow immediately. 
Let $L_1^*$ be the set of lines from $L'$ whose position is uniquely determined by $L_\mathcal{D}$,  using only the incidence structure of $P'$ and $L'$. We hope to show that $|L_1^*| \gtrsim m$.

To start, we note that, since each line from $L_1^*$ is at least $\frac{\Delta n^{2/3}}{4m^{1/3}}$-rich, the lines from $L_1^*$ must determine at least 
$$|L_1^*|\frac{\Delta n^{2/3}}{4m^{1/3}} \ge |L_\mathcal{D}|\frac{\Delta n^{2/3}}{4m^{1/3}} \ge \frac{\Delta k C_3}{8}m$$ 
incidences. If we throw away points that are on at most $1$ line from $L_1^*$, we can be throwing away at most $n \le m$ incidences. Thus, if we let $P^*$ denote the set of points from $P'$ on at least two lines from $L_1^*$, and let $k$ be sufficiently large, then we know that $L_1^*$ and $P^*$ determine at least $|L_1^*|\frac{\Delta n^{2/3}}{8m^{1/3}}$ incidences. By \cref{thm:szt}, we see
$$|L_1^*|\frac{\Delta n^{2/3}}{8m^{1/3}} \le|I(P^*,L_1^*)| \lesssim |P^*|^{2/3}|L_1^*|^{2/3} + |P^*| + |L_1^*|, $$
from which we conclude that
\begin{align}
|P^*| \gtrsim \min\left( |L_1^*|\frac{ n^{2/3}}{m^{1/3}},|L_1^*|^{1/2}\frac{n}{m^{1/2}}\right)\gtrsim_{C_3} k^{1/2}m^{1/6}n^{2/3}\label{plineq}
\end{align}
as we are assuming $m\geq n$.

In the next step, we let $L_2^* \subset L'$ be the set of lines through at least two points from $P^*$. Since each point from $P^*$ is $\frac{\Delta m^{2/3}}{4n^{1/3}}$-rich, we note that $P^*$ and $L'$ determine at least $$|P^*|\frac{\Delta m^{2/3}}{4n^{1/3}} \gtrsim_{C_3} k^{1/2}m^{5/6}n^{1/3} \geq k^{1/2}m$$ 
incidences. Discarding lines not in $L_2^*$, we lose at most $m$ incidences, and thus $P^*$ and $L_2^*$ determine at least $|P^*|\frac{\Delta m^{2/3}}{8n^{1/3}}$ incidences as long as $k$ is sufficiently large with respect to $C_3$. 
Now we run the argument for \cref{plineq} again but with the roles of points and lines reversed.
We see
$$|P^*|\frac{\Delta m^{2/3}}{8n^{1/3}} \le|I(P^*,L_2^*)| \lesssim |P^*|^{2/3}|L_2^*|^{2/3} + |P^*| + |L_2^*|. $$ 
Note that if $|L_2^*|\gtrsim |P^*|\frac{m^{2/3}}{n^{1/3}}$, then we get imediately that $|L_2^*| \gtrsim m$, and we are done.
Therefore it suffices to consider the case where $|P^*|^{2/3}|L_2^*|^{2/3}\gtrsim |P^*|\frac{m^{2/3}}{n^{1/3}}$, showing that 
\begin{align}|L_2^*| \gtrsim |P^*|^{1/2}\frac{m}{n^{1/2}}\gtrsim \min\left(|L_1^*|^{1/2}\frac{m^{5/6}}{n^{1/6}},|L_1^*|^{1/4}m^{3/4}\right)\label{finalineq}\end{align}
where for the final inequality, we used \cref{plineq}. Finally, since $L_1^*$ consists of all lines that are recoverable from $L_\mathcal{D}$ and the incidence structure of $P'$ and $L'$, we know $|L_2^*| \le |L_1^*|$. Combining this with \cref{finalineq}, we get $|L_1^*| \gtrsim m$ or $|L_1^*|\gtrsim \frac{m^{5/3}}{n^{1/3}}\geq m$, which gives us $|L_1^*|\gtrsim m$ regardless.
Using \cref{plineq}, we get $|P^*| \gtrsim n$. Note that the points in $P^*$ are uniquely determined by the points $P_\mathcal{D}$ along with the incidence structure of $P$ and $L$. Recall that $P_D$ is $r_0$-rigid for every $D\in \mathcal{D}$ for a constant $r_0$ depending on $\Delta$ and $C_2$. Combining this with \cref{prop:fullrigid} gives us that $P^*$ is $kr_0$-rigid. This completes the proof of \cref{thm:main}.
\\\\
{\bfseries Acknowledgements:}

\medskip

\noindent The authors would like to thank Zeev Dvir for helpful conversations. The research of the first author was supported in part by a Killam Doctoral Scholarship and a Four-Year Fellowship from the University of British Columbia. The research of the second author was supported in part by an NSERC Discovery grant and OTKA K 119528 grant.



\end{proof}

\bibliographystyle{abbrv}
\bibliography{Full_bib}
\end{document}